\documentclass[11pt]{article}
\usepackage{amsmath, amssymb, a4wide, graphicx, color, enumerate}
\newtheorem{theorem}{Theorem}[section]

\newtheorem{rem}[theorem]{Remark}
\newtheorem{lemma}[theorem]{Lemma}
\newtheorem{definition}[theorem]{Definition}
\newtheorem{proposition}[theorem]{Proposition}

\newtheorem{ex}[theorem]{Example}
\newenvironment{proof}{\noindent {\bf Proof.  $ $}}{\hfill $\Box$\medskip\noindent}

\newenvironment{remark}{\begin{rem} \em }{\em \end{rem}}

\newcommand{\cb}    [1]{\ensuremath{\left  \{      #1  \right \}       }}

\newcommand{\of}    [1]{\ensuremath{\left (        #1  \right )        }}
\newcommand{\ofg}   [1]{\ensuremath{\bigl (        #1  \bigr  )        }}
\newcommand{\ofgg}  [1]{\ensuremath{\biggl(        #1  \biggr )        }}
\newcommand{\scp}   [1]{\ensuremath{\left \langle  #1  \right \rangle  }}

\newcommand{\ra} {\ensuremath{\rightarrow}}
\newcommand{\rra}{\ensuremath{\rightrightarrows}}
\newcommand{\st} {\ensuremath{|\;}}

\newcommand{\cl}  {{\rm cl  \,}}
\newcommand{\Cl}  {{\rm Cl  \,}}

\newcommand{\Int} {{\rm int  \,}}

\DeclareMathOperator*{\Sup}{ Sup}
\DeclareMathOperator*{\Inf}{ Inf}

\DeclareMathOperator*{\wMax}{wMax}
\DeclareMathOperator*{\wMin}{wMin}

\newcommand{\dom}{{\rm dom\,}}

\newcommand{\I}{\mathcal{I}}

\newcommand{\B}{\mathcal{B}}
\newcommand{\A}{\mathcal{A}}

\newcommand{\D}{\mathcal{D}}

\newcommand{\LL}{\mathcal{L}}

\newcommand{\R}{\mathbb{R}}

\author{Elvira Hern\'andez \and Andreas L\"ohne \and Luis Rodr\'iguez-Mar\'in \and Christiane Tammer}
\title{Lagrange duality, stability and subdifferentials in vector optimization}
\date{December 18, 2012}

\begin{document}
\maketitle

\begin{abstract} \noindent
Lagrange duality theorems for vector and set optimization problems which are based on a consequent usage of infimum and supremum (in the sense of greatest lower and least upper bounds with respect to a partial ordering) have been recently proven. In this note, we provide an alternative proof of strong duality for such problems via suitable stability and subdifferential notions. In contrast to most of the related results in the literature, the space of dual variables is the same as in the scalar case, i.e., a dual variable is a vector rather than an operator. We point out that duality with operators is an easy consequence of duality with vectors as dual variables.
\end{abstract}

\section{Introduction} \label{sec_in}

In order to derive duality assertions in set-valued optimization one has the possibilities to use an approach via conjugates, via Lagrangian technique or an axiomatic approach. Conjugate duality statements, based on different types of perturbation of the original problem, have been derived by Tanino and Sawaragi \cite{TanSaw80}, Sawaragi, Nakayama and Tanino \cite{SawNakTan85}, L\"ohne \cite{Loe05,Loe2012-book}, Bo\c{t}, Grad and Wanka \cite{BotGraWan09}, Hamel \cite{Hamel09} and others. Lagrange duality for set-valued problems has been studied, for instance, by Luc \cite{Luc88}, Ha \cite{Troung05}, Hern{\'a}ndez and Rodr{\'{\i}}guez-Mar{\'{\i}}n \cite{HerRod07,HerRod07-1}, Li, Chen and Wu \cite{LiCheWu09}, L\"ohne \cite{Loe2012-book} and Hamel and L\"ohne \cite{HamLoe12}. An axiomatic approach was given by Luc \cite{Luc88}. Furthermore, duality assertions can be developed for different solution concepts, this means for the vector approach (cf. Luc \cite{Luc88}, Bo\c{t}, Grad and Wanka \cite{BotGraWan09}, Li, Chen and Wu \cite{LiCheWu09}), for the set-approach (cf. Kuroiwa \cite{Kur99}, Hern{\'a}ndez and Rodr{\'{\i}}guez-Mar{\'{\i}}n \cite{HerRod07}) and by using supremal and infimal sets (see Nieuwenhuis \cite{Nieuwenhuis80}, Tanino \cite{Tanino88,Tanino92}) and/or infimum and supremum in a complete lattice for the lattice approach (cf. Tanino \cite{Tanino92}, Song \cite{Song97}, \cite{Song98}, L\"ohne \cite{Loe05,Loe2012-book}, Lalitha and Arora \cite{LalAro07}, Hamel \cite{Hamel09}). 

Another important difference in constructing dual problems concerns the type of dual variables, which can be vectors like in the scalar case (cf. L\"ohne \cite{Loe05,Loe2012-book}, Bo\c{t}, Grad and Wanka \cite{BotGraWan09}, Section 7.1.3), extended vectors (cf. Hamel \cite{Hamel09}, Hamel and L\"ohne \cite{HamLoe12}) or operators (cf. Corley \cite{Corley87}, Luc \cite{Luc88}, Tanino \cite{Tanino92}, Bo\c{t}, Grad and Wanka \cite{BotGraWan09}, Section 7.1.2, Hern{\'a}ndez and Rodr{\'{\i}}guez-Mar{\'{\i}}n \cite{HerRod07}, Lalitha and Arora \cite{LalAro07} , Li, Chen and Wu \cite{LiCheWu09}). 

In the mentioned papers one can observe that it is very easy to derive weak duality statements without additional assumptions. In order to get strong duality one needs additionally convexity and certain regularity assumptions. An important approach for the formulation of regularity assumptions is the stability of the primal set-valued problem. In the literature, stability of the primal set-valued problem is formulated using the subdifferential of the minimal value map (cf. Tanino and Sawaragi \cite{TanSaw80}, Bo\c{t}, Grad and Wanka \cite{BotGraWan09}). In this paper we will introduce a subdifferential notion based on infimal sets where subgradients are vectors. A corresponding stability notion is used to prove strong duality statements. Furthermore, we will discuss the space of dual variables and explain the relations between some other approaches from the literature. Subdifferential notions for vector and set-valued problems have been investigated by many authors, see e.g., Tanino \cite{Tanino92}, Jahn \cite{Jahn04}, Bo\c{t}, Grad and Wanka \cite{BotGraWan09}, Schrage \cite{Schrage09diss}, Hamel and Schrage \cite{HamSch12}. 

Lagrange duality theory is an important tool in optimization and
there are many approaches to a corresponding theory for vector
optimization problems, see e.g. Corley \cite{Corley81,Corley87}, \cite{RubGas04}, Bo\c{t} and Wanka \cite{BotWan04}, Li and Chen \cite{LiChe97}, Jahn \cite{Jahn04}, G\"{o}pfert, Tammer, Riahi and Z\u{a}linescu \cite{GoeTamRiaZal03}, Bo\c{t}, Grad and Wanka \cite{BotGraWan09} and the references therein.

In the scalar case, the basic idea is to assign to a given constrained optimization problem (called the primal problem)
$$\mbox{(p)} \hspace{6cm} p:=\inf_{x \in S} f(x), \hspace{12cm} $$
where $f\colon X \to \overline{\R}$, a Lagrange function $L:X \times
\Lambda \to \overline \R$, where $\Lambda$ is the set of dual variables, via
$$\displaystyle{\sup_{ u^*}}\, L(x, u^*) = f(x)\quad \mbox{ for  }\quad x \in S
\subseteq X.$$
One considers the closely related pair of mutually
dual unconstrained problems
$$ \inf_{x}\sup_{ u^*} L(x, u^*) \qquad \text{ and } \qquad \sup_{ u^*} \inf_{x} L(x, u^*).$$
The dual problem is usually written as
$$\mbox{(d)} \hspace{6cm} d:=\sup_{ u^*} \phi( u^*) \hspace{12cm}  $$
where $ u^*$ is called the dual variable and $\phi:\Lambda \to
\overline
 \R,\;\;\phi( u^*) := \displaystyle{\inf_{x} }\,L(x, u^*)$ is called
the dual objective function.  If $p=d$  we say that (d) is an exact dual problem of (p). One goal in duality theory is to find a subset of $\Lambda$ as small as possible so that (d) is still an exact
dual problem of (p). The concepts of infimum and supremum play an important role in Lagrange duality. In all the vectorial approaches the problem was to find an appropriate replacement for these concepts because of non-completeness of preference orders. One possibility to reobtain the lattice structure in vector optimization is to scalarize the Lagrange function and to use the usual infimum/supremum in the complete lattice $\overline \R:= \R \cup \cb{+\infty} \cup \cb{-\infty}$ see, for instance, \cite{Jahn86, Jahn04}. Another idea was to replace the infimum/supremum by minimality/maximality notions, see e.g. Bo\c{t}, Grad and Wanka \cite{BotGraWan09}. In \cite{LoeTam07}, a vector optimization problem has been extended to a set-valued problem in order to obtain a complete lattice. Conjugate duality statements have been proven in order to demonstrate that the ideas from scalar optimization can be analogously formulated in the vectorial framework. In \cite{Loe2012-book} corresponding Lagrange duality results have been established. 

In this paper, we provide an alternative proof of the duality results given in \cite{Loe2012-book}. Our approach is motivated mainly by the paper \cite{Tanino92} by Tanino. Even though Tanino \cite{Tanino92} and related papers \cite{Song97,LiCheWu09,CheLi09} did not mention the complete lattice structure, his results are closely related to the lattice approach. In this paper we use a classical notation which is similar to the one in \cite{Tanino92} in order to emphasize the relationships. The results of this paper can be easily restated by using the infimum and supremum in a complete lattice, see e.g. \cite{Loe2012-book} for more details. Another purpose of this paper is to simplify the set of dual variables. We will use vectors rather than operators and point out that this leads to stronger duality statements. An operator variant of our duality result will be obtained as a corollary. 

This paper is organized as follows. In Section 2 we recall the concepts and results that we use later, in particular,  Lagrange duality. Section 3 is devoted to a new proof of the strong duality theorem, which is based on new stability and subgradient notions that we introduce before. In Section 4 we formulate duality assertions with operators as dual variables and show that they easily follow from the results with vectors as dual variables.
 

\section{Preliminaries}

\subsection{Infimal sets and the complete lattice of self-infimal sets}

We recall in this section the concept of an {\em infimal set} (resp. {\em supremal set}), which is due to Nieuwenhuis \cite{Nieuwenhuis80}, was extended by Tanino \cite{Tanino88}, and slightly modified with respect to the elements $\pm\infty$ in \cite{LoeTam07}. We will shortly discuss the role of the space of self-infimal sets, which was shown in \cite{LoeTam07} to be a complete lattice. 

Let $(Y,\leq)$ be a partially ordered linear topological space, where the order is induced by a convex cone $C$ satisfying
$\emptyset \neq \Int C \neq Y$. We write  $y\leq y'$ if  $y'-y\in C$ and  $y<y'$ if  $y'-y\in \Int C$. 
We denote by $\overline{Y}:=  Y \cup \cb{-\infty}\cup\cb{+\infty}$ the extended space, where the ordering is extended by the convention
\[ \forall y \in Y:\; -\infty \leq y \leq +\infty.\]
The {\em upper closure} (with respect to $C$) of $A \subseteq  \overline{Y}$ is defined \cite{LoeTam07,Loe2012-book} to be the set
$$ \Cl_+ A := \left\{ \begin{array}{lll}
               Y         & \mbox{ if } & -\infty \in A \\
               \emptyset & \mbox{ if } & A=\cb{+\infty} \\
               \cb{y \in  Y\st  \cb{y} + \Int C \subseteq A\setminus\cb{+\infty} + \Int C}& \mbox{ otherwise. } &
                      \end{array} \right.
$$
We have \cite[Proposition 1.40]{Loe2012-book} 
\begin{equation}\label{eq_clplus}
     \Cl_+ A := \left\{ \begin{array}{lll}
               Y         & \mbox{ if } & -\infty \in A \\
               \emptyset & \mbox{ if } & A=\cb{+\infty} \\
               \cl(A\setminus\cb{+\infty}+C)& \mbox{ otherwise. } &
                      \end{array} \right.
\end{equation}
The set of {\em weakly
minimal elements} of a subset $A \subseteq Y$ (with respect to
$C$) is defined by
$$ \wMin A := \cb{y \in A \st (\cb{y} -\Int C) \cap A = \emptyset},$$
We have \cite[Corollary 1.44]{Loe2012-book}
$$ \wMin \Cl_+ A \neq \emptyset \iff \emptyset \neq \Cl_+A \neq Y.$$
The {\em infimal set} of $A \subseteq \overline{Y}$ (with respect to C) is defined by
$$ \Inf A := \left\{ \begin{array}{lll}
                      \wMin \Cl_+ A  & \mbox{ if } & \emptyset \neq \Cl_+ A \neq  Y \\
		      \cb{-\infty}  & \mbox{ if } & \Cl_+ A =  Y                   \\
		      \cb{+\infty}  & \mbox{ if } & \Cl_+ A = \emptyset.
		      \end{array} \right.
$$
We see that the infimal set of $A$ with respect to $C$ coincides essentially with the set of weakly minimal elements of the set $\cl(A+C)$. Note that  if $A\subseteq Y$ then  $\wMin A=A\cap \Inf A$.

By our conventions, $\Inf A$ is always a nonempty set. If $-\infty$ belongs to $A$, we have $\Inf A = \cb{-\infty}$, in particular, $\Inf \cb{-\infty} =\cb{-\infty}$. Furthermore, one has
$\Inf \emptyset = \Inf \cb{+\infty} =  \cb{+\infty}$ and $\Cl_+ A = \Cl_+(A \cup \cb{+\infty})$. Hence $\Inf A = \Inf (A \cup \cb{+\infty})$ for all $A \subseteq \overline{Y}$.

The following properties of infimal sets seem to be essentially due to Nieuwenhuis \cite{Nieuwenhuis80}. Extended variants (but slightly different to the following ones) have been established by \cite{Tanino88}.  

\begin{proposition}[\cite{Loe2012-book}, Corollary 1.48] \label{pr}
Let $A,B \subseteq  Y$ with $\emptyset \neq \Cl_+A \neq  Y$ and $\emptyset \neq \Cl_+B \neq  Y$, then
\begin{enumerate}[(i)]
\item \label{pr_0_1} $ \Inf A =
              \cb{y \in  Y \st  y \not\in A + \Int C,\;
              \cb{y} + \Int C \subseteq A + \Int C}$,
\item \label{pr_0_2} $ A + \Int C = B + \Int C \iff \Inf A = \Inf B$,
\item \label{pr_0_3} $A + \Int C = \Inf A + \Int C$,
\item \label{pr_0_4} $\Cl_+ A = \Inf A \cup (\Inf A + \Int C)$,
\item \label{pr_0_5} $\Inf A$, $(\Inf A- \Int C)$ and $(\Inf A + \Int C)$ are disjoint,
\item \label{pr_0_6} $\Inf A \cup (\Inf A- \Int C) \cup (\Inf A + \Int C) =  Y$.
\item \label{pr_0_7} $\Inf(\Inf A + \Inf B) = \Inf(A+B)$,
\item \label{pr_0_8} $\alpha \Inf A = \Inf (\alpha A)$ for $\alpha >0$.
\end{enumerate}
For $A \subseteq \overline{Y}$ one has
\begin{enumerate}[(i)]
\item[(ix)] $\Inf \Inf A = \Inf A$, $\Cl_+ \Cl_+ A = \Cl_+ A$, $\Inf \Cl_+ A = \Inf A$, $\Cl_+ \Inf A = \Cl_+ A$,
\end{enumerate}
\end{proposition}

\begin{proposition}\label{pr_3}
Let $A_i \subset  \overline{Y}$ for $i \in I$, where $I$ is an arbitrary index set. Then
\begin{enumerate}[(i)]
\item $\displaystyle \Cl_+ \bigcup_{i\in I} A_i = \Cl_+ \bigcup_{i\in I} \Cl_+ A_i$,
\item $\displaystyle \Inf \bigcup_{i\in I} A_i = \Inf \bigcup_{i\in I} \Inf A_i$.
\end{enumerate}
\end{proposition}

Using the set $\wMax A:= -\wMin (-A)$ of weakly maximal elements of $A\subseteq  Y$, one can define likewise 
the lower closure $\Cl_-A$ and the set $\Sup A$ of supremal elements of $A \subseteq \overline{Y}$ and there are analogous statements.

We introduce the following ordering relation for sets $A,B \subseteq \overline{Y}$:
$$ A \preccurlyeq B \;:\iff\; \Cl_+ A \supseteq \Cl_+ B. $$

Given a partially ordered set $(Z,\le)$, we say that $\bar z \in Z$ is a {\em lower bound} of $A \subseteq Z$ if $\bar z \le a$ for all $a \in A$. The element $\bar z \in Z$ is called the infimum of $A\subseteq Z$ (written $\bar z = \inf A$) if $\bar z$ is a lower bound of $A$ and if $\hat z \le \bar z$ holds for every other lower bound $\hat z$ of $A$. As the ordering $\le$ is antisymmetric, the infimum, if it exists, is uniquely defined. The partially ordered set $(Z,\le)$ is called a {\em complete lattice} if every subset of $Z$ has an infimum and a supremum, where the supremum is defined analogously and is denoted by $\sup A$. 

The following result has been established in \cite{LoeTam07}, a proof can be also found in \cite[Theorem 1.54]{Loe2012-book}. We denote by $\I$ the space of all self-infimal subsets of $\overline{Y}$, where $A\subseteq \overline{Y}$ is called {\em self-infimal} if $A=\Inf A$.

\begin{theorem} [\cite{LoeTam07}]\label{th_1}
The partially ordered set $(\I, \preccurlyeq)$ provides a complete lattice. For nonempty sets $\A \subseteq \I$ it holds
$$ \inf \A = \Inf \bigcup_{A \in \A} A , \qquad \sup \A = \Sup \bigcup_{A \in \A} A.$$
\end{theorem}

This theorem allows to formulate duality statements analogous to their scalar counterparts by using the infimum and supremum in a complete lattice. An easy consequence, for instance, is that for any function $L: X \times U \rra \overline{Y}$ one has
$$ \Sup\bigcup_{u \in U} \Inf\bigcup_{x\in X} L(x,u) \preccurlyeq \Inf\bigcup_{x\in X} \Sup\bigcup_{u \in U} L(x,u).$$
Another consequence is that, for an arbitrary function $f:X\rra \overline{Y}$ and $A \subseteq B \subseteq X$, one has
$$ \Inf\bigcup_{x\in A} f(x) \succcurlyeq \Inf\bigcup_{x\in B} f(x)  
\quad\text{ and }\quad
 \Sup\bigcup_{x\in A} f(x) \preccurlyeq \Sup\bigcup_{x\in B} f(x).$$
A further useful consequence is the following result, which is a reformulation of \cite[Proposition 1.56]{Loe2012-book}.

\begin{proposition}\label{pr_2}
For nonempty subsets $\A,\B \subseteq \overline{Y}$ we have
\begin{enumerate}[(i)]
\item \label{pr_2_1}$\displaystyle \Inf\bigcup_{ A \in \A,\; B \in \B} (A + B) =  \Inf\of{\Inf \bigcup_{A\in \A} A + \Inf\bigcup_{B\in\B} B }$,
\item \label{pr_2_2}$\displaystyle \Sup\bigcup_{ A \in \A,\; B \in \B} (A + B) \preccurlyeq  \Sup \bigcup_{A\in \A} A + \Sup\bigcup_{B\in\B} B $.
\end{enumerate}
\end{proposition}

\subsection{Lagrange duality} \label{sec_lagr}

In this section we recall Lagrange duality results as given in \cite[Section 3.3.2]{Loe2012-book} for optimization problems with set-valued objective function and set-valued constraints using a notation adapted to \cite{Tanino92}.

Let $X$ be a linear space and let $U$ be a separated locally convex space. Moreover, let $\scp{U,U^*}$ be a dual pair.
Let $F :  X \rra \overline{Y}$ and $G :  X \rightrightarrows  U$ be set-valued maps. We set $\dom F:=\cb{x \in X \st F(x)\neq \cb{+\infty}}$ and  $\dom G:=\cb{x \in X \st G(x)\neq \emptyset}$. Let $D \subseteq  U$ be a proper closed convex cone with nonempty interior. We denote by $D^\circ$ the negative polar cone of $D$.

We consider the following primal problem \eqref{p}:
\begin{equation}\label{p}
 \tag{P} \hspace{3cm} \bar p:=\Inf \bigcup_{x \in S} F(x),\qquad  S:=\cb{x \in  X \st G(x) \cap -D \neq \emptyset}.
\end{equation}
Constraints of this type have been investigated by many authors, such as Borwein \cite{Borwein77};
Corley \cite{Corley87}; Jahn \cite{Jahn83};  Luc \cite{Luc88}; G\"otz and Jahn \cite{GoeJah99}; Crespi, Ginchev and Rocca \cite{CreGinRoc06}; Bo\c{t} and Wanka \cite{BotWan04}.

We assume throughout that a fixed vector 
\begin{equation}\label{eq_c}
c\in \Int C
\end{equation}
is given. Several concepts, for instance, the Lagrangian, the dual objective function and subgradients will depend on the choice of this vector $c$. Note that we do not mention this dependance explicitly.
 
The Lagrangian map of problem \eqref{p} is
defined by
\begin{equation}\label{eq_lagr}
 L:  X \times  U^* \rra \overline{Y},\qquad
   L(x, u^*) = F(x) + \Inf \bigcup_{u \in G(x)+D}\scp{ u^*,u}\cb{c}.
\end{equation}
In the special case $q=1$, $C = \R_+$, $c = 1$, the well-known Lagrangian
coincides with the Lagrangian of the scalar problem. For every choice of $c \in \Int C$ we have a different Lagrangian map and a different corresponding dual problem, but we show that weak duality and strong duality hold for all these problems.

The scalar counterpart of the following result is well known.

\begin{proposition}[\cite{Loe2012-book}, Proposition 3.23] \label{pr_r45}
Let $x \in S$, then
$$ \Sup \bigcup_{ u^* \in  U^*} L(x, u^*) = F(x).$$
\end{proposition}

We next define the dual problem. The dual objective function is defined by
$$  \phi:  U^* \rra \overline{Y}, \qquad \phi( u^*):= \Inf \bigcup_{x \in  X} L(x, u^*)$$
and the dual problem (with respect to $c \in \Int C$) associated to \eqref{p} is defined by
\begin{equation}\label{d}
 \tag{D} \bar d:=\Sup \bigcup_{ u^* \in  U^*} \phi(u^*).
\end{equation}

As shown in \cite[Theorem 3.25]{Loe2012-book}, we have weak duality. Taking into account Theorem \ref{th_1}, we get the following formulation.

\begin{theorem}[weak duality] \label{th_1w}  The problems \eqref{p} and \eqref{d} satisfy the weak duality inequality
$$\Sup \bigcup_{ u^* \in  U^*} \phi(u^*) \preccurlyeq \Inf \bigcup_{x \in S} F(x).$$
\end{theorem}

The following strong duality theorem has been proven in \cite[Theorem 3.26]{Loe2012-book} by using scalarization and a scalar Lagrange duality result. We present a reformulation based on Theorem \ref{th_1}.

\begin{theorem}[strong duality]\label{th_ld}
Let $F$ be $C$-convex, let $G$ be $D$-convex, and let
\begin{equation}\label{eq_sl5}
   G(\dom F) \cap - \Int D \neq \emptyset.
\end{equation}
Then, strong duality holds, that is, 
$$\Sup \bigcup_{ u^* \in  U^*} \phi(u^*) = \Inf \bigcup_{x \in S} F(x).$$
\end{theorem}
By strong duality and Theorem \ref{th_1} we get
\[ \Sup\bigcup_{u^*\in U^*} \Inf\bigcup_{x\in X} L(x,u^*) =\Inf\bigcup_{x\in X} \Sup\bigcup_{u^*\in U^*} L(x,u^*).\]
The next statement extends Proposition \ref{pr_r45}. Note that the additional assumption of $G(x)+D$ being a closed convex set originates from the set-valued constraints. It cannot be omitted even if the objective function would be scalar-valued, see \cite[Example 3.21]{Loe2012-book}.
\begin{proposition}[\cite{Loe2012-book}, Proposition 3.24] \label{pr_u20}
Let $F:X\rra\overline{Y}$ be a set valued map with $F(x)=\Inf F(x)\neq \cb{-\infty}$ for all $x \in X$, $\dom F \neq \emptyset$ and let the set $G(x)+D$ be closed and convex for every $x \in X$, then
$$ \Sup\bigcup_{u^* \in U^*} L(x,u^*) = \left\{ \begin{array}{clc}
                                      F(x) & \mbox{ if }&  x \in S \\
                                       \cb{+\infty} & \mbox{ else. } &
                                       \end{array} \right. $$
\end{proposition}


\section{Stability, subgradients and another proof of duality}
\label{sec_ap}

We start this section by introducing the notion of subgradient and stability for our framework. These concepts are mainly motivated by Tanino \cite{Tanino92} (see also Bo\c{t}, Grad and Wanka \cite{BotGraWan09} for numerous related results). We consider problem \eqref{p} and the related notions as introduced in Section \ref{sec_lagr}.

Let $\varphi$ be a  set-valued map  from $X \times U$ to $Y$ defined by
$$\varphi(x,u)=\left \{\begin{array}{cc}
                 F(x) & \mbox{ if }G(x)\cap (-D-u)\neq \emptyset  \\
                \emptyset   & \mbox{ else. }
               \end{array}\right.
$$
Denote by $W \colon U \rra \overline{Y}$  the  perturbation map defined by 
$$W(u)=\Inf\bigcup_{x\in X}\varphi(x,u).$$
Clearly, we have 
$$W(0)=\Inf\bigcup_{x \in S} F(x) = \bar p.$$
This leads to the following definition of a subgradient. 

\begin{definition}\label{def_subg} 
A point $u^* \in -D^\circ$ is a called {\em positive subgradient} of $W$ at
$(\bar{u},\bar{y})\in U\times Y$ with $\bar{y}\in W(\bar{x})$, written $u^* \in
\partial^+ W(\bar{u},\bar{y})$ for short, if 
$$\bar y - \scp{u^*,\bar{u}} c \in \Inf\bigcup_{u \in U}(W(u) - \scp{u^*,u} c).$$
\end{definition}
 
\begin{remark} Note that the subgradient in \cite[Definition 2.5]{LiCheWu09} is a stronger notion than the one in Definition \ref{def_subg}. In particular, in \cite[Definition 2.5]{LiCheWu09}, subgradients are operators whereas in Definition \ref{def_subg} subgradients are vectors. Furthermore, the definition of subgradients in Definition \ref{def_subg} is closely related to that given by Bo\c{t}, Grad and Wanka \cite[Definition 7.1.9 (c)]{BotGraWan09}, where the subgradients are defined using the Pareto maximum instead of the supremal set as in Definition \ref{def_subg}. However, in \cite{BotGraWan09} the so called $k$-subgradients are vectors too.
\end{remark}

\begin{lemma}\label{lem_510}
Consider the set-valued maps $\partial^+W(0,\cdot): Y \rra U^*$ and $\phi: U^* \rra \overline{Y}$. For $u^*\in U^*$ with $\phi(u^*)\subseteq Y$, one has
  $$u^* \in \partial^+W(0,y) \;\iff\; y \in \phi(u^*).$$
\end{lemma}
\begin{proof}
By definition, $u^* \in \partial^+W(0,y)$ means
$$y \in \Inf\bigcup_{u \in U}(W(u) - \scp{u^*,u} c).$$
We have
$$
\renewcommand{\arraystretch}{1.8}
\begin{array}{ll}
 \displaystyle\Inf\bigcup_{u \in U}(W(u) - \scp{u^*,u} c) &= 
 \displaystyle\Inf\bigcup_{u \in U}\Inf\bigcup_{x \in X} (\varphi(x,u) - \scp{u^*,u} c) \\ 
  &= \displaystyle\Inf\bigcup_{-u \in G(x)+D,\, x \in X} (F(x) - \scp{u^*,u} c) \\
  &= \displaystyle\Inf\bigcup_{u \in G(x)+D,\,x \in X} (F(x) + \scp{u^*,u} c) \\
  &= \displaystyle \Inf\bigcup_{x \in X} \ofgg{F(x) + \Inf\bigcup_{u \in G(x)+D} \scp{u^*,u} c}\\
  &= \displaystyle \Inf\bigcup_{x \in X} L(x,u^*) \;=\; \phi(u^*),
  \end{array}
$$  
which proves the claim.
\end{proof}

\begin{definition}
Problem \eqref{p} is called {\em stable} if $W(0)\neq \cb{+\infty}$, $W(0)\neq \cb{-\infty}$ and
$\partial^+ W(0,y)\neq \emptyset$ for all $y\in W(0)$.
\end{definition}

\begin{remark}
In Bo\c{t}, Grad and Wanka \cite{BotGraWan09} the definition of $k$-subgradients is used in order to introduce the property that the primal set-valued optimization problem is $k$-stable: The problem $(P)$ is called $k$-stable with respect to a certain set-valued perturbation map if the corresponding minimal value map is $k$-subdifferentiable at $0$.
\end{remark}

In the proof of the next theorem we use the following lemma.

\begin{lemma}\label{lem_property} 
Let $A,B \subseteq  Y$ with $\emptyset \neq \Cl_+A \neq  Y$ and $\emptyset \neq \Cl_+B \neq  Y$, then
$$A \preccurlyeq B \Leftrightarrow (A-\Int C)\cap B=\emptyset.$$
\end{lemma}
\begin{proof}
$A \preccurlyeq B$ is equivalent to $B \subseteq \Cl_+ A$. By Proposition \ref{pr} \eqref{pr_0_5}, \eqref{pr_0_6}, the latter inclusion is equivalent to $(A-\Int C)\cap B=\emptyset$.
\end{proof}

\begin{theorem}\label{th_5_1_0} If \eqref{p} is stable, then strong duality holds for \eqref{p} and \eqref{d}, that is, 
$$\Sup \bigcup_{ u^* \in  U^*} \phi(u^*) = \Inf \bigcup_{x \in S} F(x).$$
\end{theorem}
\begin{proof} We set 
$ \bar d=\Sup \bigcup_{ u^* \in  U^*}\phi(u^*)$
and
$\bar p = W(0) =  \Inf \bigcup_{x \in S} F(x)$.
By assumption, we have $W(0) \neq \cb{+\infty}$ and $W(0) \neq \cb{-\infty}$, which implies $\emptyset \subsetneq \Cl_+ W(0) \subsetneq Y$. Take some $y \in W(0)$. Since \eqref{p} is stable, there is some $u^*\in U^*$ with $y\in \phi(u^*)$ (by Lemma \ref{lem_510}). Using weak duality we get $\phi(u^*) \preccurlyeq \bar d \preccurlyeq \bar p$. Thus $\bar d \neq \cb{-\infty}$ and $\bar d \neq \cb{+\infty}$, which implies $\emptyset \subsetneq \Cl_+ \bar d \subsetneq Y$.

By weak duality, it remains to prove $\bar p \preccurlyeq \bar d$.
Taking into account Lemma \ref{lem_property}, we have to prove that $(\bar{p}-\Int C)\cap \bar d=\emptyset.$ On the contrary, suppose that  there is $y \in Y$ with $y \in (\bar{p}-\Int C)\cap \bar d$. Then there exists $z \in \bar{p} = W(0)$ and $c \in \Int C$ such that 
$y=z-c$. On the other hand, there exists $u^* \in \partial^+W(0,z)$. By Lemma \ref{lem_510}, this means $z \in \phi(u^*)$. Hence $y \in (\phi(u^*)-\Int C) \cap \bar d$. By Lemma \ref{lem_property}, this contradicts $\phi(u^*) \preccurlyeq \bar d$.
\end{proof}

\begin{theorem}\label{th_5_1} If $F$ is $C$-convex, $G$ is $D$-convex, 
\begin{equation}\label{eq_cq3}
G(\dom F) \cap  (-\Int D) \neq \emptyset,
\end{equation}
and $W(0)\neq \cb{-\infty}$, then \eqref{p}  is stable.
\end{theorem}
\begin{proof}
From \eqref{eq_cq3}, we get $W(0) \neq \cb{+\infty}$.

Let $\bar y \in W(0)$. By Lemma \ref{lem_510} we have to show that there exists $u^*\in U^*$ with $\bar y \in \phi(u^*)$.
The map $Q\colon X\rightrightarrows Y\times U$ defined by $Q(x)=(F(x),G(x))$ is $C\times D$-convex. Thus, $Q(X)+C\times D$ is a convex set. We next show that
\begin{equation}\label{eq_1_th_5_1}
\left(Q(S)+C\times D\right)\cap \Int (B\times (-D))=\emptyset
\end{equation}
where $B=\cb{\bar y}-C$. Indeed, if there exist $x'\in S$ and $(y,u)$ such that 
$$(y,u)\in \left((F(x'),G(x'))+C\times D\right)\cap \Int (B\times (-D)),$$
then $y\in (F(x')+C)\cap (\cb{\bar y}-\Int C)$ and $u\in (G(x')+D)\cap -\Int D$. Thus, $y'=\bar{y}-c'$ where $y'\in F(x')$ and $c'\in \Int C$ and $(G(x')+D)\cap (-D)\neq \emptyset$ (that is, $x'\in S$) which contradicts $\bar y \in W(0) = \Inf \bigcup_{x \in S} F(x).$

By \eqref{eq_1_th_5_1}, applying a separation theorem, there exists a pair $(y^*,u^*)\in Y^*\times U^*\setminus \{(0,0)\}$ such that
\begin{equation}\label{eq_2_th_5_1}
\scp{y^*, y} +\scp{u^*, u} \leq \scp{y^*, b}+\scp{u^*,-d}  
\end{equation}
for all $(y,u)\in Q(S)+C\times D$, $b\in B$ and $d\in D$.
We deduce that $(y^*,u^*)\in (C^\circ \times D^\circ) \setminus \{(0,0)\}.$ 
This implies
$$
\forall (y,u) \in Q(S)+C\times D:\; \scp{y^*, y} + \scp{u^*,u} \leq \scp{y^*, \bar y}. 
$$
Since, by \eqref{eq_clplus}, $\Cl_+ W(0) = \Cl_+ F(S) = \cl(F(S)+C)$, we get
\begin{equation}\label{eq_3_th_5_1}
\forall y \in \Cl_+ F(S),\; \forall u \in G(S)+D:\; \scp{y^*, y} + \scp{u^*,u} \leq \scp{y^*, \bar y}. 
\end{equation}

We show that $y^* \neq 0$. Assuming the contrary, we get $u^*\neq 0$ and, by \eqref{eq_3_th_5_1}, we have 
$$\forall u \in G(S)+D:\; \scp{u^*, u}\leq 0.$$
On the other hand, by \eqref{eq_cq3}, there exists $x \in S$ and $u' \in G(x) \cap (-\Int D)$, i.e., $\scp{u^*, u'}>0$ as  $u^* \in D^\circ\setminus \{0\}$. Since $u' \in G(x) \subseteq  G(x) + D$ this is a contradiction.

Since $y^* \in C^\circ\setminus\cb{0}$, for the fixed vector $c \in \Int C$ according to \eqref{eq_c}, we have 
$\scp{y^*,c}<0$. Without loss of generality we can assume $\scp{y^*,c}=-1$. 

Since $\bar y \in W(0) \subseteq \Cl_+ F(S)$, by \eqref{eq_3_th_5_1}, we have  $\scp{u^*, u}\leq 0$ for all $u\in G(\bar x) \subseteq G(\bar x)+D$. Since $u^*\in D^\circ$, we have $\scp{u^*,\bar u}=0$ for all $u \in G(\bar x)\cap -D$. Thus, \eqref{eq_3_th_5_1} can be written as
\begin{equation*}
\forall y \in \Cl_+ F(S),\; \forall u \in G(S)+D:\; \scp{y^*, y - \scp{u^*,u}\cb{c}} \leq \scp{y^*,\bar y}.
\end{equation*}
From weak duality, we know that $\bar y \in \Cl_+ \phi(u^*)$. Assuming that $\bar y \in \phi(u^*) + \Int C$, we obtain a contradiction to the latter inequality. Hence, by Proposition \ref{pr}, we have $\bar y \in \phi(u^*)$.
\end{proof}

\noindent
{\bf Another proof of Theorem \ref{th_ld}:} If $W(0)=\cb{-\infty}$ the statement follows from weak duality.
Otherwise it is obtained by combining Theorem \ref{th_5_1_0} with Theorem \ref{th_5_1}. \hfill $\Box$\medskip\noindent

\section{Lagrange duality with operators as dual variables}

We establish in this section another type of dual problem where the dual variables are operators rather than vectors as in problem $(\rm D)$. The usage of operators is more common in the literature (see, for instance, \cite{Corley87,LiChe97,Kur99,Luc88,HerRod07}). We will see, however, that a duality theory based on operators as dual variables is an easy consequence of the above results.

Denote by $\mathcal{L}$ the set of all linear continuous operators from $U$ to $ Y$ and by $\mathcal{L}_{+}$ the subset of all positive operators, that is, $\mathcal{L}_+ :=\{T\in \mathcal{L}\colon T(D)\subseteq
 C \}$.  Given $T\in \mathcal{L}$ and $A\subseteq U$ we write 
$T(A)=\cb{T(a) \st a \in A}$. Let $\mathcal{L}_c$ be a subset of $\mathcal{L}$ defined by $\LL_c := \{T \in \LL \st T=\scp{ u^*,\cdot} c \text{ for some } u^* \in -D^\circ  \}. $
Obviously, we have
\begin{equation}\label{eq_op22}
 \LL_c \subseteq \LL_+.
\end{equation}
Moreover, $\LL_c$ is isomorphic to $-D^\circ \subseteq U^*$.

The Lagrangian map  $L: X \times \LL \to \overline{Y}$ is defined by
\begin{equation}\label{L}
 L(x,T):= F(x) + T(G(x)).
\end{equation}
We define the dual objective function $\Phi: \LL \to \overline{Y}$ by
\begin{equation*}\label{duality-f-2}
 \Phi(T):= \Inf\bigcup_{x \in X} L(x,T).
\end{equation*}
The associated dual problem is
\begin{equation}\label{d_op}
 \tag{$\D$} \tilde d:= \Sup\bigcup_{T \in \LL_+} \Phi(T).
\end{equation}

Comparing the two dual problems \eqref{d} (with vectors as dual variables) and \eqref{d_op} (with operators as dual variables), we observe that Lagrangian \eqref{eq_lagr} for problem \eqref{d} involves the cone $D$ but Lagrangian \eqref{L} for problem \eqref{d_op} does not. On the other hand, the supremum in \eqref{d} is taken over the whole linear space $U^*$ whereas in \eqref{d_op} only the subspace $\LL_+$ of the linear space $\LL$ is considered. A reformulation of problem \eqref{d} clarifies the connection. Consider, instead of \eqref{eq_lagr}, the Lagrangian
\begin{equation}\label{eq_lagr1}
 \hat L:  X \times  U^* \rra \overline{Y},\qquad
   \hat L(x, u^*) = F(x) + \Inf\bigcup_{u \in G(x)}\scp{ u^*,u}\cb{c},
\end{equation}
and the corresponding dual objective function
\[  \hat\phi:  U^* \ra \overline{Y}, \qquad \hat\phi( u^*):= \Inf\bigcup_{x \in  X} \hat L(x, u^*).\]
\begin{lemma} The dual objective function of problem \eqref{d} can be expressed as
\begin{equation}\label{phi_c}
\phi(u^*)= \left \{\begin{array}{cl}
                                 \hat \phi(u^*) & \mbox{ if }\quad  u^* \in -D^\circ \\
                                 \cb{-\infty} & \mbox{ otherwise. }
                               \end{array}\right.
\end{equation}
\end{lemma}
\begin{proof} 
Since $c \in \Int C$, we have
\[
\Inf\bigcup_{d \in D} \scp{d,u^*} c = \left\{
 \begin{array}{cl}
   \cb{0}   & \text{ if } u^* \in -D^\circ \\
   \cb{-\infty} & \text{ otherwise. }
 \end{array}     \right.                         
\]
It follows
\[
\renewcommand{\arraystretch}{2}
  \begin{array}{ll}
    \phi(u^*) &= \displaystyle\Inf\bigcup_{x \in X} L(x,u^*) \\
              &= \displaystyle\Inf\bigcup_{x \in X} \ofgg{F(x) + \Inf\bigcup_{u \in G(x)+D} \scp{u,u^*}c }\\
              &= \displaystyle\Inf\bigcup_{x \in X} \ofgg{F(x) + \Inf\bigcup_{u \in G(x), d \in D} \ofg{\scp{u,u^*}c + \scp{d,u^*}c}} \\
              &= \displaystyle\Inf\bigcup_{x \in X} \ofgg{F(x) + \Inf\bigcup_{u \in G(x)} \scp{u,u^*}\cb{c}  + \Inf\bigcup_{d \in D}\scp{d,u^*}\cb{c}} \\
              &= \displaystyle \Inf\bigcup_{x \in X} \hat L(x,u^*) + \Inf\bigcup_{d \in D}\scp{d,u^*}\cb{c} \\
              &= \displaystyle \hat \phi(u^*) + \Inf\bigcup_{d \in D}\scp{d,u^*}\cb{c}.              
  \end{array}                      
\]
Combining the two equations, we obtain the result.
\end{proof}

As a consequence, we can define a dual problem
\begin{equation}\label{d_hat}
 \tag{$\hat D$} \hat d:=\Sup\bigcup_{ u^* \in  -D^\circ} \hat\phi(u^*),
\end{equation}
where we obviously have
\begin{equation}\label{dd}
\bar d = \Sup\bigcup_{ u^* \in  U^*} \phi(u^*) = \Sup\bigcup_{ u^* \in  -D^\circ} \hat\phi(u^*) = \hat d.
\end{equation}
Since $-D^\circ$ is isomorphic to $\LL_c$ and $\LL_c \subseteq \LL_+$, we get (using Theorem \ref{th_1})
 \begin{equation}\label{ieq_dd}
  \bar d = \Sup\bigcup_{ u^* \in  U^*} \phi(u^*) \preccurlyeq \Sup\bigcup_{ T \in \LL_+} \Phi(T) = \tilde d. 
 \end{equation} 
We next prove weak duality.

\begin{theorem}[weak duality]\label{th_1wL}  The problems \eqref{p} and \eqref{d_op} satisfy the weak duality inequality, i.e., $$\Sup\bigcup_{ T \in \LL_+} \Phi(T) \preccurlyeq \Inf\bigcup_{x\in S} F(x).$$
\end{theorem}
\begin{proof}
By Theorem \ref{th_1}, we have
\[ \Sup\bigcup_{T \in \LL_+} \Inf\bigcup_{x \in X} L(x,T) \preccurlyeq \Inf\bigcup_{x \in X} \Sup\bigcup_{T \in \LL_+} L(x,T).\]
Since $\Phi(T)= \Inf\bigcup_{x \in X} L(x,T)$, it remains to show
\[\Inf\bigcup_{x \in X} \Sup\bigcup_{T \in \LL_+} L(x,T) \preccurlyeq \Inf\bigcup_{x \in S} F(x).\]
But this follows from Proposition \ref{pr_r45} and
\[\Inf\bigcup_{x \in X} \Sup\bigcup_{T \in \LL_+} L(x,T) \preccurlyeq \Inf\bigcup_{x \in S} \Sup\bigcup_{T \in \LL_+} L(x,T),\]
which is a consequence of Theorem \ref{th_1}.
\end{proof}

Finally we obtain strong duality as a conclusion of the Lagrange duality theorem with vectors as variables.

\begin{theorem}[strong duality]
Let $F$ be $C$-convex, let $G$ be $D$-convex, and let
\[
   G(\dom f) \cap (- \Int D) \neq \emptyset.
\]
Then strong duality holds, that is,
$$\Sup\bigcup_{ T \in \LL_+} \Phi(T) = \Inf\bigcup_{x \in S} F(x).$$
\end{theorem}
\begin{proof}
From Theorem \ref{th_ld}, inequality \eqref{ieq_dd}, and  Theorem \ref{th_1wL}, we get
$$\Inf\bigcup_{x \in S} F(x) = \Sup\bigcup_{ u^* \in U^*} \phi(u^*) \preccurlyeq \Sup\bigcup_{ T \in \LL_+} \Phi(T) \preccurlyeq \Inf\bigcup_{x \in S} F(x),$$
which yields the desired equation.
\end{proof}

\bibliographystyle{abbrv}


\end{document}